\documentclass[10pt]{article}
\usepackage{graphicx}
\usepackage{latexsym}
\usepackage{float}
\usepackage{amsmath}
\usepackage{cite}
\usepackage{amssymb}
\usepackage{amsthm}
\usepackage{url}
\usepackage{tabulary}
\usepackage{booktabs}
 \usepackage{helvet}

\theoremstyle{plain}
\newtheorem{theorem}{Theorem}[section]

\newtheorem{lemma}[theorem]{Lemma}

\newtheorem{pro}[theorem]{Proposition}

\newtheorem{Qu}[theorem]{Question}
\theoremstyle{definition}
\newtheorem{definition}{Definition}[section]

\theoremstyle{remark}

\def \N {{\mathbb N}}

\def \Z {{\mathbb Z}}

\def \K {{\mathbb K}}

\def \0 {{\mathbf 0}}

\date{}

\title{A generalization of Selfridge's question}

  \author{Devendra Prasad \\ dp742@snu.edu.in}

\begin{document}

\maketitle

\begin{abstract} Selfridge asked to investigate the pairs $(m,n)$ of natural numbers for which $2^m - 2^n$ divides $x^m - x^n$ for all integers $x.$ This question was answered by different mathematicians by showing that there are only  finitely many such pairs. Let $R$ be the ring of integers of a number field $\K$ and $M_n(R)$ be ring of all $n \times n$ matrices over $R$.    In this article, we prove a generalization of Selfridge's question in the case of $M_n(R)$.
 
\end{abstract}
\textbf{keywords:}  Fixed divisors, Exponential congruences, Polynomials.
\maketitle

\section{Introduction}\label{s1}


This article is devoted to a generalization of the question asked by Selfridge.
Once he observed that $2^2-2$ divides $n^2-n, 2^{2^2}-2^2$ divides $ n^{2^2}-n^2$ and $2^{2^{2^2}}-2^{2^2}$ divides $n^{2^{2^2}}-n^{2^2}$ for all $ n \in \N.$ Motivated by this example he asked the question: for what pairs of natural numbers $m$ and $n,$ $(2^m-2^n)\mid (x^m-x^n)$ for all integers $x?$ When he observed this example and when he asked this question is not known as per our information. This question was published in the book ``Unsolved Problems in Number
Theory" by Richard Guy (see \cite{guy}, problem B47).    In 1974, Ruderman posed a similar problem

\begin{Qu}(Ruderman \cite{ruderman}) Suppose that $ m > n > 0$ are integers such that $2^m - 2^n$ divides $3^m - 3^n$. Show that $2^m - 2^n$ divides $x^m - x^n$ for all natural numbers $x.$
\end{Qu}

This famous question is called `Ruderman's problem` in the literature and is still   open. A positive solution to it will lead to the  answer of Selfridge's question. In 2011, Ram Murty and Kumar Murty \cite{MurtyMurty} proved that there are only finitely many $m$ and $n$ for which the hypothesis in the Question holds. Rundle \cite{rundle} also examined two types of generalizations of the Selfridge's problem. 

Selfridge's problem was answered by Pomerance \cite{pomerance} in 1977 by combining results of Schinzel \cite{schinzelruderman} and Velez \cite{velez}. Q. Sun and M. Zhang \cite{sunzhang} also answered Selfridge's question. Actually, there are fourteen such  pairs which are solution  of Selfridge's question and they are (1,0), (2,1), (3,1), (4,2), (5,1), (5,3), (6,2), (7,3), (8,2), (8,4), (9,3), (14,2), (15,3) and (16,4).

  Once  Selfridge's question is answered completely,  a natural question arises: what happens if we replace `2' by `3' or more generally by some other integer (other than $\pm$ 1). Bose \cite{bose} considered the  question of finding solutions of  $( b^m - b^n) \mid
a^m-a^n\ \forall\ a \in \Z  $, where $m$ and $n$ are positive integers with $m >n$ and $b$ is an integer satisfying $( b^m - b^n) \neq 0$.  He proved  that the congruence has a solution if and only if $b=2$.

Now,   a natural question crops up in our mind: what happens if there are three (or more) terms in the Selfridge's problem? More precisely, what are the tuples $(m_1, \ldots, m_k) $ with positive entries such that for a given polynomial $f(x)= \sum_{i=1}^ka_ix^{m_i} \in \Z[x]$ and given integer $b$, $f(b)$   divides $f(m)\ \forall\ m\ \in \Z,$   under reasonable conditions.

   The arguments used to answer  Selfridge's question were elementary and will  not suffice to answer this question as already pointed out by Bose (see \cite{bose}). However, the notion of the fixed divisor of a polynomial still works  . We first give a general definition of this notion.

\begin{definition}\label{def fix div} Let $A$ be a ring and $f(\underline{x}) \in A[\underline{x}]$ be a polynomial in $n$ variables. Given $\underline{S} \subseteq A^n,$ the fixed divisor of $f$ over $\underline{S},$ denoted by $d(\underline{S},f),$ is defined as the ideal of $A$ generated by the values taken by  $f$ on $\underline{S}.$  \end{definition}
 
 In the case of   $\Z$ or a Unique Factorization domain (UFD), we   manipulate the Definition \ref{def fix div} as follows and  this definition is more useful than
the above definition in this case.
 
 \begin{definition}\label{def fix div ufd}
For a polynomial    $f(x) \in \Z[x]$, its fixed divisor over $\Z$ is  defined as
 
 $$d(\Z,f) = \gcd \{ f(a): a \in \Z \}. $$
 
 \end{definition}
 
 Now we explain how this notion is helpful in the study of Selfridge's question.
 Observe that for a given $a \in \Z \backslash \{\pm 1 \}$,   $a^m - a^n \mid x^m - x^n\ \forall\ x \in \Z$ iff $a^m - a^n \mid d(\Z,f_{m,n}), $ where $f_{m,n} = x^m - x^n$. 
Let $a_1, a_2,\ldots, a_k$ be non-zero elements of $\Z$ and $C$ be set of all polynomials with coefficients $a_1, a_2,\ldots, a_k$, then $\{ d(\Z,g): g \in C\}$ is bounded  by a number which depends on $a_1, a_2,\ldots, a_k$ and is independent to the choice of $m$ and $n$ (for a proof see Vajaitu \cite{vajaituapp1}).

In the case mentioned above, the coefficients are $\pm 1,$ and hence it follows that $d(\Z,f_{m,n}) \leq M$ for some real constant $M$ and hence only finitely many pairs $(m,n)$ are possible such that $a^m - a^n \mid x^m - x^n\ \forall\ x \in \Z.$

 In 1999, Vajaitu \cite{vajaituapp1} generalized Selfridge's problem in a number ring and proved 

\begin{theorem}[Vajaitu and Zaharescu\cite{vajaituapp1}]\label{vajaitu app1} Let $R$ be a number ring of an algebraic number field and let $a_1,a_2,\ldots a_k$ and $b$ be non-zero elements of $R$. Let $b$ be a non-unit, then there are only finitely many $k$ tuples $(n_1,n_2, \ldots, n_k) \in \N^k$ satisfying the following simultaneously

$$   \sum_{i=1}^k a_ib^{n_i} \vert \sum_{i=1}^k a_ix_i^{n_i} \quad \forall\ x \in R $$ and
$$ \sum_{i \in S} a_ib^{n_i} \neq 0\ \forall S \subseteq \lbrace 1,2,\ldots,k \rbrace.$$
\end{theorem}

 Vajaitu and Zaharescu also strengthened the conclusion of Theorem \ref{vajaitu app1} for the ring of integers in a specific number field.
 
 \begin{theorem}[Vajaitu and Zaharescu\cite{vajaituapp1}]\label{thm:number ring}

 Let $R$ be   the ring of rational integers $\Z$ or the ring of integers in
an imaginary quadratic number field and let $a_1,a_2,\ldots a_k$ and $b$ be non-zero elements of $R$.
Then there are only finitely many elements β in $R$ for which there exist $k$-tuples $(n_1,n_2, \ldots, n_k) \in \N^k$, not all zero, satisfying   the following simultaneously

$$   \sum_{i=1}^k a_ib^{n_i} \vert \sum_{i=1}^k a_ix_i^{n_i} \quad \forall\ x \in R $$ and
$$ \sum_{i \in S} a_ib^{n_i} \neq 0\ \forall S \subseteq \lbrace 1,2,\ldots,k \rbrace.$$
 
 \end{theorem}

 Theorem \ref{vajaitu app1}   generalizes the question of Selfridge to the case of a number ring. In 2004, Choi and Zaharescu \cite{choi} generalized  Theorem \ref{vajaitu app1} in the case of $n$-variables. 

\begin{theorem}[Choi and Zaharescu \cite{choi}]\label{choi zahe} Let $R$ be the ring of integers in an algebraic number
field and let $b_1,b_2,\ldots b_n$  be non-zero non-unit elements of $R$.  
Let $a_{i_1,\ldots,i_n} \in R\ \forall\ 1 \leq i_1 \leq k_1,\ldots,1 \leq i_n \leq k_n, $ then there are only finitely many $n$ tuples $(\mathbf{m}_1,\mathbf{m}_2, \ldots, \mathbf{m}_n) \in \N^{k_1} \times \N^{k_2}\times \cdots \times \N^{k_n}$ satisfying the following simultaneously
$$   \sum_{i_1=1}^{k_1} \cdots \sum_{i_n=1}^{k_n} a_{i_1,\ldots,i_n} b_1^{m_{1i_1}} \cdots b_n^{m_{ni_n}} \vert \sum_{i_1=1}^{k_1} \cdots \sum_{i_n=1}^{k_n} a_{i_1,\ldots,i_n}x_1^{m_{1i_1}} \cdots x_n^{m_{ni_n}} \forall\   \underline{x} \in R^n $$ where $\mathbf{m}_j = (m_{j1},\ldots m_{jk_j})$ and
$$ \sum_{(i_1, \ldots, i_n) \in S}   a_{i_1,\ldots,i_n} b_1^{m_{1i_1}} \cdots b_n^{m_{ni_n}} \neq 0\  $$
for all non-empty $S \subseteq \lbrace 1,2,\ldots, k_1 \rbrace \times \cdots \times \lbrace 1,2,\ldots,k_n \rbrace $.
\end{theorem}

  Choi and Zaharescu also generalized Theorem \ref{thm:number ring} in this setting.
We write the statement for the sake of completeness.

\begin{theorem}[Choi and Zaharescu \cite{choi}]\label{choi zahe1} Let $R$  be the ring of rational integers $\Z$ or the ring of integers in
an imaginary quadratic number field.  Fix $n $   and choose nonzero elements  $a_{i_1,\ldots,i_n} \in R\ \forall\ 1 \leq i_1 \leq k_1,\ldots,1 \leq i_n \leq k_n, $. Then there
are only finitely many $n$-tuples $(b_1,b_2,\ldots b_n)$ with $b_j \in R, j = 1, \ldots , n$
for which there exists $(\mathbf{m}_1,\mathbf{m}_2, \ldots, \mathbf{m}_n) \in \N^{k_1} \times \N^{k_2}\times \cdots \times \N^{k_n}$ with none of the tuples $(\mathbf{m}_1,\mathbf{m}_2, \ldots, \mathbf{m}_n)$   having all the components equal to zero  satisfying the following  simultaneously
  
$$   \sum_{i_1=1}^{k_1} \cdots \sum_{i_n=1}^{k_n} a_{i_1,\ldots,i_n} b_1^{m_{1i_1}} \cdots b_n^{m_{ni_n}} \vert \sum_{i_1=1}^{k_1} \cdots \sum_{i_n=1}^{k_n} a_{i_1,\ldots,i_n}x_1^{m_{1i_1}} \cdots x_n^{m_{ni_n}} \forall\   \underline{x} \in R^n $$ where $\mathbf{m}_j = (m_{j1},\ldots m_{jk_j})$ and
$$ \sum_{(i_1, \ldots, i_n) \in S}   a_{i_1,\ldots,i_n} b_1^{m_{1i_1}} \cdots b_n^{m_{ni_n}} \neq 0\  $$
for all non-empty $S \subseteq \lbrace 1,2,\ldots, k_1 \rbrace \times \cdots \times \lbrace 1,2,\ldots,k_n \rbrace $.
\end{theorem}

In this article, we consider a generalization of Theorem \ref{vajaitu app1}, Theorem \ref{thm:number ring}  Theorem \ref{choi zahe}   and Theorem \ref{choi zahe1}  in the case when the ring under consideration is that of    $n \times n$ matrices over a number ring. We denote by $M_n(A)$, the ring of $n \times n$ matrices over the given ring $A$. We   use the fixed divisor of a polynomial as our tool in the generalization. In the case of the ring of the matrices over a ring, a reasonable definition of the fixed divisor of a polynomial is suggested by Prasad, Rajkumar and Reddy (\cite{prasadsurvey}, Sec. 7) with suitable justification.

 

 \begin{definition} For a polynomial $f \in M_n(A)[x],$ 
 its   fixed divisor over $M_n(A)$ (or $d(M_n(A),f) $) is defined as the ideal in $ A$ generated by all the entries of    $f(C)\ \forall\ C \in M_n(A). $ 
 \end{definition}

Observe that here the fixed divisor is not an ideal of the ring  $M_n(A)$ as usual. This definition is helpful in the study of fixed divisors and related topics.  For a given matrix $M \in M_n(\K)$ for any number field $\K$, recall that the norm  
 $$  \Vert M \Vert  = (\sum_{i,j}^n \mid m_{ij} \mid^2)^{\tfrac{1}{2}}, $$
 makes the space $(M_n(\K), \Vert \cdot \Vert)$ a Banach algebra.
We suggest the following generalization of Selfridge's question.
\begin{Qu}\label{main qu}   Let $A_1,A_2,\ldots,A_k,B$ be non-zero matrices in $M_n(R)$ and $B$ satisfies the following
\begin{itemize}

\item[(A.1)] The ideal generated by    $B$  
is not the whole ring.

\item[(A.2)]  All the  eigenvalues of $B^*B$ are non-zero and      their modulii   is either strictly  less than $1$  or  strictly   greater than $1$. Here $B^*$ is  the conjugate transpose of  $B$.


\end{itemize}
Then, for how many tuples     $ (m_1,m_2,\ldots, m_k) \in \N^k $, the following are satisfied simultaneously
\begin{enumerate}
\item[(B.1)]     $\Sigma_{i \in S} A_iB^{m_i} \neq 0\ \forall\ S \subseteq \{1,2,\ldots,k \} ,$  
\item[(B.2)]  the ideal generated by  $\sum_{i =1}^k A_iB^{m_i}   $ contains the ideal generated by  $\{ \sum_{i =1}^k A_iC^{m_i} : C \in M_n(R) \}? $
\end{enumerate}

\end{Qu}

We know that each ideal of $M_n(A)$ is of the form  $M_n(I)$  for some ideal $I \subseteq A$. Also, for each $I \subseteq A$,   $M_n(I)$ is an     ideal of $M_n(A)$.  For given ideals $I$ and $J$ of $A$, the condition $M_n(I) \subseteq M_n(J)$ is equivalent to saying that $I \subseteq J.$ For a matrix $M \in M_n(A)$ we denote by $I_M$, the ideal generated by all entries of $M$ in $A$.  Hence, we have to find the number of tuples $ (m_1,m_2,\ldots, m_k) \in \N^k $, for which $I_{f(B)} \supseteq d(M_n(A),f)$ where $f= \sum_{i =1}^k A_ix^{m_i}.$

\medskip

The structure of the paper is as follows. In Section \ref{secction bounds} we give bounds for fixed divisors by using combining the arguements of Vajaitu and Zaharescu and fixed divisors. Indeed, our work is motivated by the work of Vajaitu and Zaharescu. In Section \ref{section generalization} we answer our question by proving Theorem \ref{main theorem}. Finally, in Section \ref{secction further} we suggest further generalization of our theorems in the case of several variables when the underlying ring is still $M_n(R)$.

  



  

\section{Bounds for fixed divisors}\label{secction bounds}

 We fix the notations for the whole paper. Let $\N$ denote the set of natural numbers as usual.  For a given tuple $\textbf{m}= (m_1, m_2,\ldots,$ $m_k) \in \N^k$, $m$   denotes the maximum of $m_i$ where $i=1,2,\ldots,k$. $R$ denotes a number ring and norm of an ideal $I\subseteq R$, is denoted by $N(I)$ and is the cardinality of the residue class ring $R/I$. Norm of an element is the norm of the ideal generated by the element.  

 In order to prove our main theorem we need several lemmas.   With all the notations as in Question \ref{main qu}, we prove the following lemma, in which we consider the case when modulus of each eigenvalue of $ B^*B $ is  strictly  less than one. The other case can be handled by considering  considering $B^{-1}$.

   \begin{lemma}\label{lem lower bound} Let $\mathbf{m}=(m_1, \ldots,m_k) \in \N^k$ and $m$ be the supremum of the components of $\mathbf{m}$. Then there exist constants $c$ and $d$  independent to $\textbf{m}$ such that 
   $$\Vert \sum_{i=1}^k A_iB^{m_i} \Vert \geq c\mid d \mid ^m.$$
   
   \end{lemma} 
   
   \begin{proof}
   We claim that $\Vert \sum_{i=1}^k A_iB^{m_i-m} \Vert   \geq c$, for a constant $c$.  Denoting  the difference $m-m_i$ by $n_i$ for $i=1,2,\ldots,k$, we have to show  $\Vert \sum_{i=1}^k A_iB^{-n_i} \Vert \geq c$ If this is false then there would exist a sequence $(n_{1,r}, n_{2,r}, \ldots, n_{k,r})$ of natural numbers with min $\lbrace n_{1,r}, n_{2,r}, \ldots, n_{k,r}\rbrace =0$ for each $r$, such that when $r$ tends to infinity
 $\Vert \sum_{i=1}^k A_iB^{-n_{i,r}} \Vert$ tends to zero. Let $A \subseteq \{ 1,2,\ldots,k \}$ be the largest subset such that for  each $r \in A$ there   exist a natural number $b_r $ and an infinite  sequence   $M$ such that $n_{i,r} =b_r$ for each $r \in A$ and $i \in M$. 
 Then, we    have the following inequality
 
 \begin{equation}\label{eqn ineq}
   \Vert \sum_{i=1}^k A_iB^{-n_{i,r}} \Vert   \leq  \Vert \sum_{i \in A}  A_i   B^{-b_{i}} \Vert +  \Vert \sum_{i \in A^c}  A_i   B^{-n_{i,r}} \Vert          .  
   \end{equation}
 
Here $A^c$ denotes the complement of $A$ with respect to the set $\{1,2,\ldots,k \}$. Now we recall that for any matrix $M \in M_n(\K),  \Vert M \Vert \leq \sqrt{n} \vert \lambda \vert,$ where $\lambda$ is the  eigenvalue of  $M^*M$ with maximum modulus. Hence, the second term in the Eq. (\ref{eqn ineq}) becomes
 
 $$  \Vert \sum_{i \in A^c}  A_i   B^{-n_{i,r}} \Vert   \leq  \sum_{i \in A^c} \Vert   A_i \Vert \Vert   B^{-n_{i,r}} \Vert \leq  \sum_{i \in A^c} \Vert   A_i \Vert \vert   \lambda \vert^{n_{i,r}} ,$$
 where $\lambda$ is an eigenvalue of $B^*B$ with maximum modulus. By assumption $\vert   \lambda  \vert <1, $   $  \Vert \sum_{i \in A^c}  A_i   B^{-n_{i,r}} \Vert$ tends to zero as $r$ tends to infinity. Now we rewrite the  the left hand side   of Eq. (\ref{eqn ineq}) as

 $$\Vert \sum_{i \in A}  A_i   B^{-b_{i}}  +   \sum_{i \in A^c}  A_i   B^{-n_{i,r}} \Vert, $$
  which   tends to zero as  $r$ tends to infinity. Observe that $  \Vert \sum_{i \in A^c}  A_i   B^{-n_{i,r}} \Vert$  also tends to zero as $r$ tends to infinity leading to the conclusion that $$  \Vert \sum_{i \in A}  A_i   B^{-b_{i}} \Vert \rightarrow\ 0  .$$

 This implies that $\sum_{i \in A}  A_i   B^{-b_{i}}=0$ , which is a contradiction to B.1 of the Question \ref{main qu}
 

Hence,  $\Vert \sum_{i=1}^k A_iB^{m_i-m} \Vert \geq c$ and using the fact that $\Vert XY \Vert \leq \Vert X \Vert \Vert Y \Vert\ \forall\ X,$  $ Y \in M_n(R)$ we get the desired result.
 
   \end{proof}
    
  For a matrix $M=[m_{ij}]_{n \times n}$, define $\sigma (M)=  [\sigma( m_{ij})]_{n \times n}$ for a given automorphism $\sigma$ of $\K$. Then it can be seen that $\sigma(\Vert M \Vert^2)= \Vert \sigma( M ) \Vert^2$.
In the above lemma, taking product over all the  conjugates of $\Vert f(B) \Vert$  we get $N(\Vert f(B) \Vert^2) \geq c'\vert d'\vert^m$, where  $c'$ and $d'$ are new constants.   It can be seen that $\Vert f(B) \Vert^2  \in I_{f(B)}$  so there exist an ideal $J\subseteq R$ such that  $I_{f(B)} J = (\Vert f(B) \Vert^2)$. Hence, $N(I_{f(B)} J)= N(\Vert f(B) \Vert^2) = N(I_{f(B)})N(J)$. Now we prove that $N(J)$ is also bounded above.

 \begin{lemma} Let $J$ be the ideal such that $I_{f(B)} J = (\Vert f(B) \Vert^2)$. Then there exist constants $c'$ and $d'$ not depending on $\textbf{m}$  such that $N(J) \leq c' \mid d' \mid^m.$
 
   \end{lemma}

   \begin{proof}
   Proof follows by the triangle inequality  
   
 $$\Vert \sum_{i=1}^k A_iB^{m_i} \Vert \leq \sum_{i=1}^k \Vert A_i \Vert \Vert B \Vert^{m_i}.$$

   
    \end{proof}

 Combining these two lemmas with the observation $  N(I_{f(B)}) = \tfrac{ N(\Vert f(B) \Vert^2)}{N(J)} $, we get the following proposition.
 \begin{pro} There exist non-zero constants $c_1$ and $d_1$  depending on $A_1,A_2$  $, \ldots, A_k, B$ and not depending on $\textbf{m} \in \N^k,$ such that  $N(I_{f(B)}) \geq c_1 \vert d_1 \vert^m$, where $m$ is the maximum component of $\mathbf{m}$.
 
 \end{pro}
  
  We end this section with the following lemma.
  
  \begin{lemma}\label{lem upper} For a polynomial $f =\sum_{i =1}^k A_ix^{m_i} \in M_n(R)[x] $, there exist constants $c_3,c_4,c_5$ and $c_6$ not depending on $\mathbf{m}$ such that 
  
   $$N(d(M_n(R),f)) \leq c_3 N(a_1)^{c_4}\mathrm{exp} \left(c_5 m^{\tfrac{c_6}{\mathrm{log( log} m)}}\right)$$
   
  \end{lemma}
  
  \begin{proof}

   We construct a polynomial $  \sum_{i=0}^k a_ix^{m_i} \in R[x] $ where each $a_i$ is (1,1) th (or some fixed position) entry of the matrix $A_i$. Now observe that 
   $$ d(M_n(R),f) \supseteq  d(R,f)\supseteq d(R,g).$$ Hence, we have the following
   $$ N(d(M_n(R),f)) \leq  N(d(R,f)) \leq N(d(R,g)).$$ Now using the fact that
      $  N(d(R,g)) \leq c_3 N(a_1)^{c_4}\mathrm{exp} \left(c_5 m^{\tfrac{c_6}{\mathrm{log( log} m)}}\right)$ (see \cite{vajaituapp1}, Prop. 2), where $c_3,c_4,c_5$ and $c_6$ are constants not depending on $\mathbf{m}$, we conclude that the lemma holds.  
     \end{proof}

     \section{ A Generalization of Selfridge's question in the case of ring of matrices}\label{section generalization}

   We start this section with our   main theorem.

  \begin{theorem}\label{main theorem} Let $f= \sum_{i =1}^k A_ix^{m_i} \in M_n(R)[x]$ be a polynomial and $B \in M_n(R)$ be a matrix satisfying A.1 and A.2 of the Question \ref{main qu}. Then there are finitely many tuples in $\N^k$ such that B.1 and B.2 of   Question \ref{main qu} are satisfied.   

\end{theorem}

 \begin{proof}
 
 We know that $d(M_n(R),f)$ is the ideal in $R$ generated by all the entries of $f(A)\ \forall\ A\ \in M_n(R)  .$ Also, the condition B.2 says  that $I_{f(B)}$ contains the ideal in $R$ generated by all the entries of $f(A)\ \forall\ A\ \in M_n(R)  $. Hence B.2  is equivalent to saying  that $d(M_n(R),f) \subseteq I_{f(B)}.$ This implies that  $N(d(M_n(R),f)) \geq N(I_{f(B)}).$   Invoking Lemma \ref{lem lower bound} and Lemma \ref{lem upper} we get the following inequality for  $N(d(M_n(R),f))$  
 
 $$ c_1 \mid d \mid^m     \leq N(d(M_n(R),f)) \leq c_3   N(a_1)^{c_4}\mathrm{exp} \left(c_5 m^{\tfrac{c_6}{\mathrm{log( log} m)}}\right).$$ 

    Now we   compare the bounds of $ N(d(M_n(R),f))$ to get the desired result that $m$ is bounded above. 
    
   

   
 \end{proof}

 We strengthen Theorem \ref{main theorem} in the case when $R$ is a special domain to get a generalization of Theorem \ref{thm:number ring}
 \begin{theorem}\label{main theorem2}
 Let $R$ be   the ring of rational integers $\Z$ or the ring of integers in
an imaginary quadratic number field and let $A_1,A_2,\ldots A_k$ and $B$ be non-zero elements of $M_n(R)$.
Then there are only finitely many elements $B$ in $M_n(R)$ for which there exist $k$ tuples $(n_1,n_2, \ldots, n_k) \in \N^k$, not all zero, satisfying   the following simultaneously
\begin{enumerate}
\item  the ideal generated by  $    \sum_{i=1}^k A_iB^{n_i}$ contains the ideal generated by

$ \sum_{i=1}^k A_ix^{n_i} \quad \forall\ x \in M_n(R) ;$
\item $ \sum_{i \in S} A_iB^{n_i} \neq 0\ \forall S \subseteq \lbrace 1,2,\ldots,k \rbrace.$
\end{enumerate}

\end{theorem}

\begin{proof}

The value of the Vandermonde matrix
\[
 \begin{vmatrix}
1 & 2^{m_1} & 2^{2m_1}  & \dots & 2^{(k-1)m_1} \\ 
1 & 2^{m_2} & 2^{2m_2}  & \dots & 2^{(k-1)m_2} \\
 \vdots &    \vdots &  \vdots &  \ddots &  \vdots &\\
1 & 2^{m_k} & 2^{2m_k}  & \dots &2^{(k-1)m_k}
\end{vmatrix}
\]

cannot be zero. Consequently, the matrices

 $$f(2^jI)= \sum_{i =1}^k A_i(2^jI)^{m_i}\ \mathrm{where}\ I\ \mathrm{is\ the\ identity\ matrix\ and}\ 0 \leq j \leq k-1,$$
cannot be the zero matrix for all $0 \leq j \leq k-1,$ as this would imply that some combination of the columns of the above Vandermonde matrix is zero. Now we have

\begin{equation}\label{eqn bound 1}
 \Vert f(2^jI)  \Vert  \leq \sum_{i =1}^k \Vert A_i \Vert  {n}^{ \tfrac{m_i}{2}} 2^{(k-1)m_i}     <  2^{(k-1)m_1} {n}^{ \tfrac{m_1}{2}} \sum_{i =1}^k     \Vert A_i \Vert .   
\end{equation}

If $\Vert B \Vert \geq  \tfrac{2 \sum_{i =2}^k  \Vert A_i \Vert  }{\Vert A_1\Vert},$ then we have

\begin{equation*}
 \Vert  \sum_{i =2}^k    A_i  B^{m_i}  \Vert \leq \sum_{i =2}^k \Vert A_i\Vert \Vert  B\Vert^{m_i} \leq \tfrac{\Vert A_1 \Vert \Vert B \Vert^{m_1} }{ 2 } , 
\end{equation*}
 
which implies 
\begin{equation}\label{eqn bound 2}
  \tfrac{\Vert A_1 \Vert \Vert B \Vert^{m_1} }{ 2 } \leq \Vert f(B) \Vert .
\end{equation}
The condition  $f(B) \supseteq f(2^jI)$ implies $\Vert f(B) \Vert  \leq \Vert f(2^jI) \Vert$ for the ring under consideration. If $\Vert B \Vert$ is large enough then the lower bound in the  Eq. (\ref{eqn bound 2}) is greater than the upper bound in the Eq. (\ref{eqn bound 1}), which is a contradiction. Since only finitely many elements in $M_n(R)$ can have a given norm, hence our proof is done. 

\end{proof}

We end this section with the following question.

\begin{Qu} Can we find the tuples explicitely which are the answer to the Question \ref{main qu}?
\end{Qu}

\section{Generalization to several variables}\label{secction further}

   We can extend Theorem \ref{main theorem} and Theorem  
   \ref{main theorem2}  to the multivariate case by induction on the number of variables to get a generalization of Theorem \ref{choi zahe} and Theorem \ref{choi zahe1}. Here we state the results formally for the sake of completeness and omit the proofs. 
   
   For given tuples $\mathbf{m}, \mathbf{n} \in \N^k$,   $\textbf{m} \leq \textbf{n}$ means each entry of the tuple $\textbf{m}$ is less than or equal to the corresponding entry of the tuple $\textbf{n}.$ Also, we denote the tuple $(1,1, \ldots, 1) \in \N^r$  by $\mathbf{1}.$  
   \begin{theorem}
   Let $f(\underline{x}) = \sum_{\mathbf{i}=\mathbf{1}}^{\mathbf{k}}A_{\mathbf{i}} x^{m_1i_1}_1 x^{m_2i_2}_2 \ldots x^{m_ri_r}_r \in M_n(R)[\underline{x}]$ be a polynomial in $r$ variables and $B_1,B_2, \ldots, B_r$ be   matrices satisfying the following:
   \begin{itemize}
   \item The ideal generated by    $B_i$  
is not the whole ring for all $i=1,2, \ldots, r$.
   \item All the  eigenvalues of $B_i^*B_i$ are non-zero and      their modulus   is either strictly  less than $1$  or  strictly   greater than $1$ for all $i=1,2, \ldots, r$. Here $B_i^*$ is  the conjugate transpose of  $B_i$.

\end{itemize}   
 Then, there are only finitely many tuples $(\mathbf{m}_1, \mathbf{m}_2 \ldots, \mathbf{m}_r) \in \N^{k_1} \times \N^{k_2} \times \ldots \times \N^{k_r}  $ where $\mathbf{m}_j=(m_{j1}, m_{j2}, \ldots, m_{jk_j
   })$ such that
   
   \begin{enumerate}

     \item  $  \Sigma_{\mathbf{i} \in S } A_{\mathbf{i}} B^{m_1i_1}_1 B^{m_2i_2}_2 \ldots B^{m_ri_r}_k \neq 0,$     for any nonempty set $S$ of  
      $\lbrace 1, 2, \ldots, k_1 \rbrace \times
    \cdots \times \lbrace 1, 2, \ldots, k_r \rbrace ,$

   \item the ideal generated by $f(B_1,B_2, \ldots, B_r)$ contains the ideal generated by $\{ f(A_1,A_2, \ldots, A_r)\ \forall\ (A_1,A_2, \ldots, A_r) \in M_n(R)^r  \}.$
    \end{enumerate}

   \end{theorem}
   
   Likewise, we can make an analogue of the Theorem \ref{main theorem2} as follows

   \begin{theorem}
   Let $R$ be   the ring of rational integers $\Z$ or the ring of integers in
an imaginary quadratic number field and let $\{ A_{\mathbf{i}} : \mathbf{0} \leq \mathbf{i} \leq \mathbf{k} \}$  and $\{ B_i: 1 \leq i \leq k \}$ be non-zero elements of $M_n(R)$.
Then there are only finitely many elements $\{ B_i: 1 \leq i \leq k \}$ in $M_n(R)$ for which there exist   tuples $(\mathbf{m}_1, \mathbf{m}_2 \ldots, \mathbf{m}_r) \in \N^{k_1} \times \N^{k_2} \times \ldots \times \N^{k_r}  $, not all zero, satisfying   the following simultaneously

\begin{enumerate}

     \item $\Sigma_{\mathbf{i} \in S } A_{\mathbf{i}} B^{m_1i_1}_1 B^{m_2i_2}_2 \ldots B^{m_ri_r}_k \neq 0,$
    for any nonempty set $S$ of    $\lbrace 1, 2, \ldots, k_1 \rbrace \times
    \cdots \times \lbrace 1, 2, \ldots, k_r \rbrace ,$
      
   \item the ideal generated by $f(B_1,B_2, \ldots, B_r)$ contains the ideal generated by $\{ f(A_1,A_2, \ldots, A_r)\ \forall\ (A_1,A_2, \ldots, A_r) \in M_n(R)^r  \}.$
    \end{enumerate}
   
   \end{theorem}
   
      \section{Acknowledgement} We thank Dr. Krishnan Rajkumar and Saurav Vikash Chatterjee for their valuable suggestions. We are also thankful to the anonymous referee whose comments improved this paper.

\end{document}